\documentclass{amsart}
\usepackage{amsfonts,amssymb,amscd,amsmath,enumerate,verbatim,calc}
\usepackage[all]{xy}

\newcommand{\CM}{Cohen-Macaulay}

\newcommand{\wrt}{with respect to}

\newcommand{\Ic}{\mathcal{I} }

\newcommand{\m}{\mathfrak{m} }

\newcommand{\R}{\mathcal{R} }
\newcommand{\Z}{\mathbb{Z} }
\newcommand{\D}{\mathcal{D} }
\newcommand{\K}{\mathcal{K} }

\newcommand{\rt}{\rightarrow}

\newcommand{\ov}{\overline}

\newcommand{\Tc}{\mathcal{T} }

\newcommand{\Xb}{\mathbf{X}}
\newcommand{\Zb}{\mathbf{Z}}
\newcommand{\Yb}{\mathbf{Y}}

\newcommand{\Pb}{\mathbf{P}}
\newcommand{\Qb}{\mathbf{Q}}

\newcommand{\cone}{\operatorname{cone}}

\newcommand{\depth}{\operatorname{depth}}
\newcommand{\projdim}{\operatorname{projdim}}
\newcommand{\injdim}{\operatorname{injdim}}
\newcommand{\proj}{\operatorname{proj}}

\newcommand{\CMS}{\operatorname{\underline{CM}}}
\newcommand{\thick}{\operatorname{thick}}

\newcommand{\Syz}{\operatorname{Syz}}
\newcommand{\md}{\operatorname{mod}}

\newcommand{\Hom}{\operatorname{Hom}}
\newcommand{\Ext}{\operatorname{Ext}}
\newcommand{\Tor}{\operatorname{Tor}}

\theoremstyle{plain}

\newtheorem{theorem}{Theorem}[section]
\newtheorem{corollary}[theorem]{Corollary}
\newtheorem{lemma}[theorem]{Lemma}
\newtheorem{proposition}[theorem]{Proposition}

\theoremstyle{definition}

\newtheorem{s}[theorem]{}

\theoremstyle{remark}

\begin{document}

\title[Dominant rings]{Derived functors and Hilbert polynomials over dominant Cohen-Macaulay  rings }
\author{Tony~J.~Puthenpurakal}
\date{\today}
\address{Department of Mathematics, Indian Institute of Technology Bombay, Powai, Mumbai 400 076, India}

\email{tputhen@gmail.com}
 \begin{abstract}
Let $(A,\mathfrak{m})$ be a Cohen-Macaulay local ring of dimension $d$, residue field $k$ and let $I$ be an $\mathfrak{m}$-primary ideal. Assume $k$  is infinite. Let $N$ be a perfect $A$-module of dimension $t \geq 1$. Let $M$ be a MCM $A$-module.  The function $n \rt \ell(\Tor^A_j(M, N/I^{n+1}N))$  is of polynomial type and let  $r_{I, N}^j(M)$  be its degree. In general we have
$r_{I, N}^j(M) \leq r_{I,N}^2(\Syz^d_A(k))$  for $j \geq d + 2$. If $A$ is also dominant in the sense of Takahashi \cite{T2} and $M$ is non-free then we show $\limsup_{j \rt \infty} r_{I, N}^j(M) = r_{I,N}^2(\Syz^d_A(k))$. We prove an analogous result when $V$ is a Cohen-Macaulay module of finite injective dimension and the polynomial type function $n \rt \ell(\Ext_A^j(M, V/I^{n+1}V))$
\end{abstract}
 \maketitle
\section{introduction}
Let $(A,\m)$ be a \CM \ local ring  of dimension $d \geq 1$ and let $I$ be an $\m$-primary ideal. If $D$ is an $A$-module of finite length then $\ell(D)$ denotes its length. Let $M$ be a maximal \CM \ (= MCM) $A$-module and let $N$ be a perfect $A$-module of dimension $t \geq 1$. The function $n \rt  \ell(\Tor^A_j(M, N/I^{n+1}N))$ is of polynomial type say of degree
$r_{I, N}^j(M)$, see \cite[Corollary 4]{Theo}.  In general we have
$r_{I, N}^j(M) \leq r_{I,N}^2(\Syz^d_A(k))$  for $j \geq d + 2$, see \ref{bound}. Here $\Syz^d_A(k)$ denotes the $d^{th}$-syzygy of $k$, see \ref{bound}.
In \cite[Theorem 18]{P1} we proved that
$r_{\m, A}^1(M)  = d-1$ for any non-free MCM $A$-module. Set the degree of the zero polynomial to be $-1$. It was also shown that if $I$ is a parameter ideal then $r_{I, A}^1(M)  = -1$, see \cite[Remark 20]{P1}.
In general it is a difficult question to determine $r_{I, N}^j(M)$ and the answer is known only for a few classes of ideals and modules, see \cite[3.5]{KT} for  some examples.
The fact that $r_{\m, A}^1(M) = d-1$  for non-free MCM's has  important consequences in the study of associated graded modules (\wrt \ $\m$) of MCM $A$-modules, see \cite{P3}.

Recently we proved that the properties of the singularity category of $A$ have a significant bearing on the properties of $r_{I, N}^j(M)$. In \cite{P4} we proved that if $A$ is a hypersurface ring then there exists an integer $r_I \geq -1$ depending only on $A$ and $I$ such that $r_{I, A}^1(M)  = r_I$ for all non-free MCM $A$-modules $M$. The property we used was that if $\CMS(A)$ is the stable category of $A$ then any non-zero thick subcategory of $\CMS(A)$ contains $\CMS^0(A)$,  the thick subcategory consisting of all MCM modules which are free on the punctured spectrum of $A$,
see \cite[6.6]{T}.

Let $R$ be a  Noetherian local ring and let $\D^b(R)$ be the bounded derived category of finitely generated $R$-modules. Let $\D_{sg}(R)$ be the singularity category of $R$, i.e., the Verdier quotient of $\D^b(R)$ by the thick subcategory  generated by $R$. When $R$ is Gorenstein then $\D_{sg}(R)$ is isomorphic to the stable category of MCM $R$-modules, \cite[4.4.1]{Bu}. Recently Takahashi defined a rich class of local rings called dominant rings, see \cite{T2}. A local ring $R$ is called dominant if the thick closure of any non-zero element in $\D_{sg}(R)$ contains the residue field of $R$.
We prove:
\begin{theorem}
\label{main-tor} Let $(A,\m)$ be a dominant \CM \ local ring of dimension
$d$ and let $I$ be an $\m$-primary ideal of $A$. Assume the residue field of $A$ is infinite. Let $N$ be a perfect $A$-module of dimension $t \geq 1$.  Let $M$ be a non-free MCM $A$-module. Then we have
 $\limsup_{j \rt \infty} r_{I, N}^j(M) = r_{I,N}^2(\Syz^d_A(k))$.
\end{theorem}

\s We also have a dual result to Theorem \ref{main-tor}. Let $V$ be a \CM \ $A$-module of finite injective dimension and of dimension $t$. The function \\  $n \rt  \ell(\Ext_A^j(M, V/I^{n+1}V))$ is of polynomial type say of degree
$s_{I, V}^j(M)$, see \cite[Corollary 4]{Theo}.  In general we have
$s_{I, V}^j(M) \leq s_{I,V}^1(\Syz^d_A(k))$  for $j \geq d + 1$, see \ref{bound}. We prove
\begin{theorem}
\label{main-ext} Let $(A,\m)$ be a dominant \CM \ local ring of dimension
$d$ and let $I$ be an $\m$-primary ideal of $A$.  Assume the residue field of $A$ is infinite. Let $V$ be a \CM \ $A$-module of dimension $t \geq 1$ and of finite injective dimension. Let $M$ be a non-free MCM $A$-module.  Then we have
 $\limsup_{j \rt \infty} s_{I, V}^j(M) = s_{I,V}^1(\Syz^d_A(k))$.
\end{theorem}

\s A complete intersection is dominant if and only if it is a hypersurface ring, see \cite[1.1]{T2}. There exist complete intersection rings which are not hypersurfaces for which Theorems \ref{main-tor} and Theorem \ref{main-ext} do not hold, see \cite[15.6]{P5}.

We now describe in brief the contents of this paper. In section two we discuss some preliminaries  that we need. In section three we give our bounds on $r_{I, N}^j(M)$  when $j \geq d + 2$ and 
on $s_{I, V}^j(M)$  when $j \geq d + 1$.
 In section four we study polynomials associated to complexes in the bounded derived category and the singularity category. In section five we give a proof of Theorems \ref{main-tor} and Theorem \ref{main-ext}.
\section{Preliminaries}
We discuss a few preliminary results that we need.

\s \label{rachel} Let $A$ be a  Noetherian local ring  of dimension $d \geq 1$ and let $I$ be an ideal of $A$. Let $\R = A[It]$ be the Rees algebra of $I$.
We need the following well-known and easily proved result.
\begin{proposition}\label{rach-prop} (with setup as in \ref{rachel}). Let $M$ be a finitely generated $A$-module. Let $E$ be a graded $\R$-module. Then
\begin{enumerate}[\rm (1)]
  \item The $A$-modules $\Tor^A_i(M, E)$ and $\Ext_A^i(M, E)$ are graded $\R$-modules for all $i \geq 0$.
\item If $E = \bigoplus_{n \in \Z} E_n$ then $\Tor^A_i(M, E) = \bigoplus_{n \in \Z} \Tor^A_i(M, E_n)$ and $\Ext_A^i(M, E) = \bigoplus_{n \in \Z}\Ext^i_A(M, E_n)$ for all $i \geq 0$.
  \item If $E$ is a finitely generated $\R$-module then $\Tor^A_i(M, E)$ and $\Ext_A^i(M, E)$ are finitely generated  $\R$-modules for all $i \geq 0$.
  \item Let $0 \rt E^\prime \rt E \rt E^{\prime \prime} \rt 0$ be an exact sequence of $\R(I)$-modules. Then it induces an  exact sequence of graded $\R$-modules
  \begin{enumerate}[\rm (a)]
    \item  \begin{align*}
           \cdots    &\rt\Tor^A_i(M, E^\prime) \rt \Tor^A_i(M, E) \rt \Tor^A_i(M, E^{\prime, \prime}) \\
             &\rt \Tor^A_{i-1}(M, E^\prime) \rt \cdots
           \end{align*}
\item  \begin{align*}
           \cdots    &\rt\Ext_A^i(M, E^\prime) \rt \Ext_A^i(M, E) \rt \Ext_A^i(M, E^{\prime, \prime}) \\
             &\rt \Ext_A^{i+1}(M, E^\prime) \rt \cdots
           \end{align*}
  \end{enumerate}
\end{enumerate}\qed
\end{proposition}

\s \label{rach-l} Let $U$ be a finitely generated $A$-module. Consider $L^I(U) = \bigoplus_{n \geq 0}U/I^{n+1}U$. We note that the inclusion $U[It] \rt U[t]$ is an inclusion of $\R$-modules. Its co-kernel is $L^I(U)(-1)$ and so $L^I(U)$ has a structure of $\R$-module.
Note $L^I(U)$ is \emph{not} a finitely generated $\R$-module if $I \neq A$.
We note that $\Tor^A_i(M, L^I(U)) = \bigoplus_{n \geq 0}\Tor^A_i(M, U/I^{n+1}U$ and $\Ext^i_A(M, L^I(U)) = \bigoplus_{n \geq 0}\Ext^i_A(M, U/I^{n+1}U)$ for all $i \geq 0$.

\begin{proposition}
\label{bella}(with hypotheses as in \ref{rachel}) Further assume $A$ is \\ Cohen-Macaulay and $I$ is a proper ideal of $A$. Let $U, V$ be finitely generated $A$-modules with $\projdim U$ and $\injdim V$ finite. Then $\Tor^A_i(M, L^I(U))$ and \\ $\Ext^i_A(M, L^I(V))$ are finitely generated $\R$-modules for all $i > d$.
\end{proposition}
\begin{proof}
  We have an exact sequence of $\R$-modules $0 \rt U[It] \rt U[t] \rt L^I(U)(-1) \rt 0$. We note that $\Tor^A_i(M, U[t]) = \bigoplus_{n \geq 0}\Tor^A_i(M, U)  = 0$ for $i > d$.
  We apply $M \otimes -$ to this exact sequence.
  It follows that
  $\Tor^A_i(M, L^I(U)(-1)) \cong \Tor^A_{i -1}(M, U[It])$ for $i \geq d +2$ as $\R$-modules and  $\Tor^A_{d+1}(M, L^I(U)(-1)) $ is an $\R$-submodule of $\Tor^A_{d}(M, U[It])$. So
  $\Tor^A_i(M, L^I(U))$  are finitely generated $\R$-modules for all $i > d$, see \ref{rach-prop}..

  We have an exact sequence of $\R$-modules $0 \rt V[It] \rt V[t] \rt L^I(V)(-1) \rt 0$. We note that $\Ext_A^i(M, V[t]) = \bigoplus_{n \geq 0}\Ext_A^i(M, V)  = 0$ for $i > d$.
  We apply $\Hom_A(M,  -)$ to this exact sequence.
  It follows that
  $\Ext_A^i(M, L^I(V)(-1)) \cong \Ext_A^{i +1}(M, V[It])$ for $i \geq d +1$.
  So $\Ext_A^i(M, L^I(V))$
    are finitely generated $\R$-modules for all $i > d$, see \ref{rach-prop}..
\end{proof}

\s\label{lila} Now assume that $A$ has infinite residue field and that $I$ is $\m$-primary.
\begin{enumerate}
  \item If $U$ is a finitely generated  graded $\R$-module then there exists $xt \in \R_1$ which is $U$ filter regular, i.e., $(0 \colon_U xt)_n = 0$ for $n \gg 0$. In fact there exists
a non-empty Zariski open set $W$ of $I/\m I$ such that if $x\in I$ be such that the image of $x$ in $ I /\m I$ is in $W$ then $xt$ is $U$-filter regular.
  \item Let $N$ be a finitely generated $A$-module. An element $x \in I$ is said to be $I$-superficial \wrt \ $N$ if there exists $c$  and $n_0$ such that $(I^{n+1}N \colon x)\cap I^cN = I^n N$ for all $n \geq n_0$. Superficial elements exist as the residue field of $A$ is infinite. In fact there exists
a non-empty Zariski open set $W$ of $I/\m I$ such that if $x\in I$ be such that the image of $x$ in $ I /\m I$ is in $W$ then $x$ is $I$-superficial \wrt \ $N$.
\item Assume $\depth N > 0$. Then if  $x \in I$ is $I$-superficial \wrt \ $N$ then $x$ is $N$-regular. Furthermore $(I^{n+1}N \colon x) = I^n N $ for all $n \gg 0$.
\item If $E_i$ for $i = 1, \ldots, c$ are finitely generated graded $\R$-modules and if $N$ is a finitely generated $A$-module then there exists $x \in I$ such that  $x \in I$ is $I$-superficial \wrt \ $N$ and $xt \in \R_1$ which is $E_i $ filter regular  for $i = 1, \ldots, c$. This follows from (1) and (2).
  \end{enumerate}

  \s \label{asymp-ko} Let $N$ be a finitely generated $A$-module. Then either $\projdim N/I^{n+1}N < \infty $ for all $n \gg 0$ or $\projdim N/I^{n+1}N  = \infty $ for all $n \gg 0$. This follows from the fact that the function  $n \rt \ell(\Tor^A_{d+1}(k, N/I^{n+1}N))$ is of polynomial type, see \cite[Theorem 2]{Ko}. Similarly $\injdim N/I^{n+1}N < \infty $ for all $n \gg 0$ or $\injdim N/I^{n+1}N  = \infty $ for all $n \gg 0$. This follows from the fact that the function  $n \rt \ell(\Ext_A^{d+1}(k, N/I^{n+1}N))$ is of polynomial type, see \cite[Theorem 2]{Ko}.

\s We  recall some results of Theodorescu.  Let $F_N(I) = \bigoplus_{n \geq 0}I^nN/\m I^nN$  the fiber-cone of $I$ \wrt \ $N$ and let  $\lambda_N(I) = \dim F_N(I)$. We note that $\lambda_N(I) \leq \dim N$ and if $I$ is $\m$-primary  then $\lambda_N(I) = \dim N$.
The first result is
\begin{theorem}\label{complex}(\cite[Proposition 3]{Theo}):
Let $(A, \m)$ be  a Noetherian local ring. Let $\mathcal{C} \colon F_2 \rt F_1 \rt F_0$ be a complex with $F_1, F_0$ finitely generated $A$-modules. Let
$N$ be another finitely generated $A$-module and let $I$ be an ideal of $A$.
Assume $H_1(\mathcal{C} \otimes N/I^nN)$ has finite length for all $n \geq 1$. Then
\begin{enumerate}[\rm (1)]
  \item The function $n \rt \ell(H_1(\mathcal{C} \otimes N/I^n N))$ of polynomial type, say of degree $u_N$.
  \item We have $u_N \leq \max \{ \dim H_1(\mathcal{C} \otimes N), \lambda_N(I) -1 \}$.
  \item If $\dim H_1(\mathcal{C} \otimes N)\geq \lambda_N(I) $ then the inequality in $\rm{(2)}$ becomes an equality.
\end{enumerate}
\end{theorem}

As a corollary Theodorescu obtains:
\begin{corollary}
\label{tor-ext-deg}(\cite[Corollary 4]{Theo}:) Let $(A, \m)$ be  a Noetherian local ring. Let $M, N$ be finitely generated $A$-modules and let $I$ be an ideal in $A$. Fix $i \geq 0$.
\begin{enumerate}[\rm (I)]
  \item Assume $\ell(\Tor^A_i(M, N/I^n N))$ is finite for all $n \geq 1$. Then
  \begin{enumerate}[\rm (i)]
    \item The function $n \rt \ell(\Tor^A_i(M, N/I^n N))$ is of polynomial type say of degree $t_I(M, N)$.
    \item $t_I(M, N) \leq \max \{ \dim \Tor^A_i(M, N), \lambda_N(I) - 1\}$.
    \item If $\dim \Tor^A_i(M, N) \geq \lambda_N(I)$ then the inequality in $\rm{(ii)}$ becomes an equality.
  \end{enumerate}
  \item Assume $\ell(\Ext^i_A(M, N/I^n N))$ is finite for all $n \geq 1$. Then
  \begin{enumerate}[\rm (i)]
    \item The function $n \rt \ell(\Ext^i_A(M, N/I^n N))$ is of polynomial type say of degree $e_I(M, N)$.
    \item $e_I(M, N) \leq \max \{ \dim \Ext^i_A(M, N), \lambda_N(I) - 1\}$.
    \item If $\dim \Ext^i_A(M, N) \geq \lambda_N(I)$ then the inequality in $\rm{(ii)}$ becomes an equality.
  \end{enumerate}
\end{enumerate}
\end{corollary}

\s \emph{Triangulated categories:}\\
We use \cite{N} for notation on triangulated categories. However we will assume that if $\mathcal{C}$ is a triangulated category then $\Hom_\mathcal{C}(X, Y)$ is a set for any objects $X, Y$ of $\mathcal{C}$.

\s\label{const-thick} Let $\Tc$ be a triangulated category with shift functor $\sum$. Let $\Ic_1, \Ic_2$ be two subcategories of $\Tc$. We denote by $\Ic_1 * \Ic_2$ the full subcategory  of $\Tc$ consisting of objects $M$  such that there is a  triangle  $M_1 \rt M \rt M_2 \rt \sum M_1$ with $M_i \in \Ic_i$.
Let $\Ic$ be a full subcategory of $\Tc$. By $\langle \Ic  \rangle$ we denote the smallest subcategory of $\Tc$ containing $\Ic$  which is closed under finite direct sums, direct summands, shifts and isomorphisms. Set  $\Ic_1 \diamond \Ic_2 = \langle  \Ic_1 * \Ic_2 \rangle$.

Set $\langle \Ic \rangle_0 = 0$. Then inductively define $\langle   \Ic   \rangle_i = \langle \Ic \rangle_{i-1} \diamond  \langle \Ic  \rangle.$
It is clear that $\langle \Ic \rangle_{i-1} $ is a subcategory of  $\langle   \Ic   \rangle_i$ for all $i \geq 1$.
Set
$$ \langle \Ic\rangle_{\infty} = \bigcup_{i \geq 0} \langle   \Ic   \rangle_i.$$

\s \label{thick} Let $\Ic$ be a subcategory of $\Tc$. By $\thick(\Ic)$ we mean the intersection of all thick subcategories of $\Tc$ containing $\Ic$. It can be shown that
$ \thick(\Ic) =  \langle \Ic\rangle_{\infty}$.

\section{A bound of $r_{I, N}^j(M)$ and $s_{I, V}^j(M)$}
\s \label{setup} \emph{Setup:} In this section we assume $A$ is \CM \ of dimension $d \geq 1$ and with an infinite residue field $k$. Let $I$ be an $\m$-primary ideal of $A$. Set $\R = A[It]$ the Rees algebra of $I$. We assume that $N$ is a perfect $A$-module of dimension $t \geq 1$. We also assume $V$ is a \CM \ $A$-module of dimension $v \geq 1$ and of finite injective dimension.
 Let $M$ be a finitely generated $A$-module.  The function $n \rt  \ell(\Tor^A_j(M, N/I^{n+1}N))$  is of polynomial type say of degree
$r_{I, N}^j(M)$. By \ref{tor-ext-deg} we get that $r_{I, N}^j(M) \leq \dim N - 1$ for $j \geq d + 1$.
The function   $n \rt  \ell(\Ext_A^j(M, V/I^{n+1}V))$ is of polynomial type say of degree
$s_{I, V}^j(M)$. By \ref{tor-ext-deg} we get that $s_{I, V}^j(M) \leq \dim V - 1$ for $j \geq d + 1$. The purpose of this section is to prove the following two results.

\begin{theorem}
\label{equi-residue}(with setup as in \ref{setup}). We have
\begin{enumerate}[\rm (1)]
  \item $r_{I, N}^j(k) = r_{I, N}^{d+2}(k) $ for $j \geq d +2$.
  \item $s_{I, V}^j(k) = s_{I, V}^{d+1}(k) $ for $j \geq d +1$.
\end{enumerate}
\end{theorem}

We also prove the following result in this section.
\begin{theorem}
\label{bound}(with setup as in \ref{setup}). We have
\begin{enumerate}[\rm (1)]
  \item $r_{I, N}^j(M) \leq r_{I, N}^{d+2}(k) $ for $j \geq d +2$.
  \item $s_{I, V}^j(M) \leq s_{I, V}^{d+1}(k) $ for $j \geq d +1$.
\end{enumerate}
\end{theorem}

The proof of Theorems \ref{equi-residue} and \ref{bound} is by induction on $\dim N = t \geq 1$ and on $\dim V = v \geq 1$. To enable the induction when $t \geq 2$ or $v \ge 2$ we need the following result.
\begin{lemma}
\label{mod-filter-sup}(with setup as in \ref{setup}). We have
\begin{enumerate}[\rm (1)]
  \item Fix $j_0 \geq d + 3$. Assume $t \geq 2$. Then
  there exists $x \in I$ which is $I$-superficial \wrt  \ $N$ and $xt$ is $\Tor^A_j(M, L^I(N))$-filter regular  for $j = d +1,\ldots, j_0$.
   Set $\ov{N} = N/x N$. Note $\ov{N}$ is perfect of dimension $t -1$. In particular we have exact sequences for $j = d+2, \ldots, j_0$ and $n \gg 0$
      $$ 0 \rt \Tor^A_j(M, N/I^n N) \rt \Tor^A_j(M, N/I^{n + 1} N) \rt \Tor^A_j(M, \ov{N}/I^{n+1}\ov{N}) \rt 0.$$
  \item Fix $j_0 \geq d + 2$. Assume $v \geq 2$. Then there exists $y \in I$ which is $I$-superficial \wrt  \ $V$ and $yt$ is $\Ext_A^j(M, L^I(V))$-filter regular  for $j = d +1,\ldots, j_0 + 1$. Set $\ov{V} = V/y V$. Note $\ov{V}$ is \CM \ of dimension $t-1$ and of finite injective dimension. In particular we have exact sequences for $j = d+1, \ldots, j_0$ and $n \gg 0$
      $$ 0 \rt \Ext_A^j(M, V/I^n V) \rt \Ext_A^j(M, V/I^{n + 1} V) \rt \Ext_A^j(M, \ov{V}/I^{n+1}\ov{V}) \rt 0.$$
\end{enumerate}
 \end{lemma}
\begin{proof}
(1) By \ref{bella},  $\Tor^A_i(M, L^I(N))$ are finitely generated $\R$-modules for all $i > d$. By \ref{lila}  there exists $x \in I$ which is $I$-superficial \wrt  \ $N$ and $xt$ is $\Tor^A_j(M, L^I(N))$-filter regular  for $j = d +1,\ldots, j_0$. As $N$ is \CM \ of positive dimension we get that $x$ is $N$-regular and $(I^{n+1}N \colon x) = I^nN$ for $n \gg 0$. say $n \geq n_0$.  We have an exact sequence of $\R$-modules
 \[
 0 \rt \bigoplus_{n = 0}^{n_0 -1}\frac{(I^{n+1}N \colon x)}{I^nN} \rt L^I(N)(-1) \xrightarrow{xt} L^I(N) \rt L^I(\ov{N}) \rt 0.
 \]
This induces short exact sequences for $n \geq n_0$
\[
0 \rt N/I^{n}N \xrightarrow{\alpha_n} N/I^{n + 1}N  \xrightarrow{\rho_n} \ov{N}/I^{n+1}\ov{N} \rt 0.
\]
Therefor for $n \geq n_0$ we have an exact sequence for all $j$
$$ \rt \Tor^A_j(M, N/I^{n}N) \xrightarrow{\alpha_n^j(M)} \Tor^A_j(M, N/I^{n + 1}N)\xrightarrow{\rho_n^j(M)} \Tor^A_j(M, \ov{N}/I^{n + 1}\ov{N}) \rt  $$
Here $\alpha_n^j(M) = \Tor^A_j(M, \alpha_n)$ and $\rho_n^j(M) = \Tor^A_j(M, \rho_n)$.
We note that $\alpha_n^j$ is nothing but the $n^{th}$ component of the map $\Tor^A_j(M, L^I(N))(-1) \xrightarrow{xt} \Tor^A_j(M, L^I(N))$ and so is injective for $j = d+ 1, \ldots, j_0$ and $n \gg 0$ as  $xt$ is $\Tor^A_j(M, L^I(N))$-filter regular  for $j = d +1,\ldots, j_0$.
It follows that for $j  =  d+2, \ldots, j_0$ and $n \gg 0$
we have exact sequences
      $$ 0 \rt \Tor^A_j(M, N/I^n N) \rt \Tor^A_j(M, N/I^{n + 1} N) \rt \Tor^A_j(M, \ov{N}/I^{n+1}\ov{N}) \rt 0.$$

(2)  By \ref{bella}, $\Ext^i_A(M, L^I(V))$ are finitely generated $\R$-modules for all $i > d$. By \ref{lila}  there exists $x \in I$ which is $I$-superficial \wrt  \ $V$ and $xt$ is $\Ext_A^j(M, L^I(U))$-filter regular  for $j = d +1,\ldots, j_0 + 1$. Rest of the arguments are similar to (1).
\end{proof}

\s \label{fin-dim} We note the following two obvious facts. Recall the degree of the zero polynomial is $-1$.
\begin{enumerate}
  \item $\projdim N/I^{n + 1}N < \infty$ for all $n \gg 0$ if and only if $r_{I, N}^j(k) = -1 $ for \emph{any} $j \geq d +1$.
  \item $\injdim V/I^{n+1}V  < \infty $ for all $n \gg 0$ if and only if $s_{I, V}^j(k) = -1 $ for \emph{any} $j \geq d +1$
\end{enumerate}

We first give
\begin{proof}[Proof of Theorem \ref{equi-residue}] (1) We prove the result by induction on $t = \dim N$. Fix $j_0 \geq d + 2$.
By \ref{setup} we have $r_{I, N}^{j}(k) = -1 $ or zero for all $j \geq d + 1$.
By \ref{fin-dim}(1) it follows that $r_{I, N}^{j_0}(k) = -1 $ if and only if $r_{I, N}^{d + 2}(k) = -1$ if and only if $\projdim A/I^n$ is finite for all $n \gg 0$.
The assertion follows when $t = 1$.

Now assume that $\dim N = t \geq 2$ and the assertion holds for all perfect modules $L$ of dimension  $t - 1$. By \ref{fin-dim}(1) it follows that $r_{I, N}^{j_0}(k) = -1 $ if and only if $r_{I, N}^{d + 2}(k) = -1$. So assume that $\projdim N/I^nN = \infty$ for all $n \gg 0$. Thus $r_{I, N}^j(k) \geq 0$ for all $j \geq d + 1$. By \ref{mod-filter-sup}(1)
it follows that there exists $x$ an $I$-superficial sequence \wrt \ $N$ such that $xt$ is $\Tor^A_j(M, L^I(N))$-filter regular  for $j = d +1,\ldots, j_0$. Note $\ov{N} = N/xN$ is perfect $A$-module of dimension $t-1$. By the exact sequence in \ref{mod-filter-sup}(1) we get that $r_{I,N}^{j_0}(k) = r_{I, \ov{N}}^{j_0}(k) + 1$ and $r_{I, N}^{d+2}(k) = r_{I, \ov{N}}^{d+2}(k) + 1$. The result follows by induction.

(2) This can be proved on the same lines as (1). We have to use \ref{fin-dim}(2)  and \ref{mod-filter-sup}(2).
\end{proof}

Next we give
\begin{proof}[Proof of Theorem \ref{bound}]
(1) Fix $j_0 \geq d + 2$
We note that we have nothing to show if $r_{I, N}^{j_0}(M) = -1$. Also if $r_{I, N}^{j_0}(M) \geq 0$ then necessarily $\projdim N/I^{n+1}N = \infty$ for all $n \gg 0$. So $r_{I,N}^{d+2}(k) \geq 0$.
This proves the result when $\dim N = 1$ as the possible degrees are $-1$ and $0$, see \ref{setup}.  We prove the result by induction on $t = \dim N$. The case $t = 1$ holds as discussed before. Now assume $\dim N = t \geq 2$ and the result is known for perfect modules of dimension $t-1$. We will assume that $r_{I, N}^{j_0}(M) \geq 0$ as otherwise we have nothing to show. It follows that $\projdim N/I^{n+1}N = \infty$ for all $n \gg 0$. So $r_{I,N}^{d+2}(k) \geq 0$. We may choose $x$ which is $I$-superficial \wrt \ $N$ and $xt$ is $\Tor^A_j(M, L^I(N))$-filter regular  for $j = d +1,\ldots, j_0$, Furthermore we can also assume that $xt$ is $\Tor^A_j(k, L^I(N))$-filter regular  for $j = d +1, d+2$. Note $\ov{N} = N/xN$ is perfect of dimension $t-1$. By the exact sequence in \ref{mod-filter-sup}(1) we get that $r_{I,N}^{j_0}(M) = r_{I, \ov{N}}^{j_0}(M) + 1$ and $r_{I, N}^{d+2}(k) = r_{I, \ov{N}}^{d+2}(k) + 1$. The result follows by induction.

(2) This can be proved on the same lines as (1). We have to use \ref{fin-dim}(2)  and \ref{mod-filter-sup}(2).
\end{proof}
\section{Polynomials associated to complexes in the bounded derived category and the singularity category}
In this section $(A,\m)$ is a Noetherian local ring and $I$ is an $\m$-primary ideal in $A$. Also $N$, $V$ are finitely generated $A$-modules
\s \emph{(chain)-Complexes:}
 We will consider chain complexes over $A$
\[
\Xb \colon  \cdots \rt \Xb_i \xrightarrow{\partial_i} \Xb_{i-1} \xrightarrow{\partial_{i-1}} \Xb_{i-2}\rt \cdots
\]
with $\partial_{i-1} \circ \partial_i = 0$ for all $i$. By $H_i(\Xb)$ we denote the $i^{th}$ homology of $\Xb$. By the $m^{th}$ shift  of $\Xb$ we mean the chain complex $\Xb[m]$ with $\Xb[m]_i = \Xb_{i+m}$ and differential $\partial[m]_{i} = (-1)^m\partial_{i+m}$.

\s By $C^b(\md A)$ we mean the category of bounded chain complexes of finitely generated left $A$-modules. Let $\K^b(A)$ be the corresponding homotopy category.
Furthermore let $\D^b(A)$ be its derived category. We note that $\K^b( A)$ and $\D^b(A)$ are triangulated categories with shift functor $[-1]$. Let $C^{-,b}(\proj A)$ be the category of bounded above complexes of finitely generated projective (free as  $A$ is local) $A$-modules with finitely generated (total) homology and let $K^{-,b}(\proj A)$ be the corresponding homotopy category. It is well-known that we have an equivalence of triangulated categories
$ \K^{-,b}(\proj A) \cong  \D^b(A)$. We will work with $ \K^{-,b}(\proj A)$. By $\D_{sg}(A)$ we denote the singularity category of $A$, i.e. $$\D_{sg}(A) = \K^{-,b}(\proj A)/\thick A.$$

 \s \label{p-bound} Let $\Xb \in \K^{-,b}(\proj A)$. For $n \geq 0$ we consider the chain complexes $\Xb \otimes N/I^{n+1}N$ and the co-chain complexes $\Hom_A(\Xb, V/I^{n+1}V)$.
  Fix $i \in \Z$.
  \begin{enumerate}
    \item  By Theorem \ref{complex} it follows that the function $n \rt \ell(H_i(\Xb \otimes N/I^{n+1}N))$ is of polynomial type and of degree bounded by $\dim N$.
    \item We have
    \begin{align*}
      \Hom_A(\Xb, V/I^{n+1}V) & \cong \Hom_A(\Xb, A)\otimes V/I^{n+1}V \\
       &\cong (\Hom_A(\Xb, A)\otimes V)\otimes A/I^{n+1}  \\
       &\cong \Hom_A(\Xb, V)\otimes A/I^{n+1}.
    \end{align*}
 We have $\Hom_A(\Xb, V)$ is a complex of finitely generated $A$-modules. It follows from \ref{complex} that the function $n \rt \ell(H^i(\Hom_A(\Xb, V/I^{n+1}V)$ is of polynomial type and of degree bounded by $\dim A$.
     \end{enumerate}

    The following lemma can be proved easily.
\begin{proposition}(with setup as above)
\label{iso} Let $\Xb, \Yb \in \K^{-,b}(\proj A)$. If $\Xb \cong \Yb$ then for all $i \in \Z$ and for all $n \geq 0$ we have isomorphism of complexes
\begin{enumerate}[\rm (1)]
  \item  $\Xb \otimes N/I^{n+1}N \cong \Yb \otimes N/I^{n+1}N$.
  \item $\Hom_A(\Xb, V/I^{n+1}V) \cong \Hom_A(\Yb, V/I^{n+1}V)$
\end{enumerate}
\end{proposition}

 \s\label{defn} Let $\Xb \in \K^{-,b}(\proj A)$. For $i \in \Z$  define $\alpha_{I, N}^i(\Xb) = \deg (n \rt \ell(H_i(\Xb \otimes N/I^{n+1}N)))$ and $\beta_{I, V}^i(\Xb) = \deg (n \rt \ell(H^i(\Hom_A(\Xb, V/I^{n+1}))))$. The integers  $\alpha_{I, N}^i(\Xb)$ and $\beta_{I, V}^i(\Xb)$ are bounded. Define
 \begin{enumerate}
   \item  $\alpha_{I, N}(\Xb) = \limsup_{i \rt \infty} \alpha_{I, N}^i(\Xb). $
   \item $\beta_{I, V}(\Xb) = \limsup_{i \rt \infty}  \beta_{I, V}^i(\Xb)$.
 \end{enumerate}

 We need the following result.
 \begin{proposition}\label{triangle}
 (with setup as above)
 Let $\Xb, \Yb, \Zb \in \K^{-,b}(\proj A)$.
  Assume we have a triangle $\Xb \xrightarrow{f} \Yb \rt \Zb \rt \Xb[-1]$ in $\K^{-,b}(\proj A)$.Then
\begin{enumerate}[\rm (1)]
\item For all $i \in \Z$  and for all $n \geq 0$ we have a long exact sequence in homology
\begin{align*}
  \cdots  &\rt  H_i(\Xb\otimes N/I^{n+1}N) \rt  H_i(\Yb\otimes N/I^{n+1}N)  \rt  H_i(\Zb\otimes N/I^{n+1}N)  \\
  &\rt  H_{i-1}(\Xb\otimes N/I^{n+1}N) \rt \cdots
\end{align*}
  \item For all $i \in \Z$ we have
  $$\alpha_{I, N}^i(\Yb) \leq \max\{ \alpha_{I, N}^i(\Xb),  \alpha_{I, N}^i(\Zb)\}.$$
  \item  We have
  $$\alpha_{I, N}(\Yb) \leq \max\{ \alpha_{I, N}(\Xb),  \alpha_{I, N}(\Zb)\}.$$
  \item  For all $i \in \Z$  and for all $n \geq 1$ we have a long exact sequence in cohomology
\begin{align*}
  \cdots  &\rt  H^i(\Hom(\Zb, V/I^{n}V)) \rt  H^i(\Hom(\Yb, V/I^{n}V))  \rt  H^i(\Hom(\Xb, V/I^{n}V))  \\
  &\rt H^{i+1}(\Hom(\Zb, V/I^{n}V))\rt \cdots
\end{align*}
  \item For all $i \in \Z$ we have
  $$\beta_{I, V}^i(\Yb) \leq \max\{ \beta_{I, V}^i(\Xb),  \beta_{I, V}^i(\Zb)\}.$$
  \item  We have
  $$\beta_{I, V}(\Yb) \leq \max\{ \beta_{I, V}(\Xb),  \beta_{I, V}(\Zb)\}.$$
\end{enumerate}
 \end{proposition}
 \begin{proof}
   (1) Set $\Zb^\prime = \cone(f)$. We note that $\Zb \cong \Zb^\prime$. We have a short exact sequence of complexes $0 \rt \Yb \rt \Zb^\prime \rt \Xb[-1] \rt 0$. As all the modules involved are free we get an exact sequence of complexes for all $m \geq 0$,
   $$  0 \rt \Yb\otimes N/I^{n+1}N \rt \Zb^\prime\otimes N/I^{n+1}N \rt \Xb\otimes N/I^{n+1}N[-1] \rt 0.  $$
The long exact sequence in homology and Proposition \ref{iso} gives the result.

(2) This follows from (1).

(3)  This follows from (2).

(4) This is similar to (1).

(5) This follows from (4).

(6)  This follows from (5).
 \end{proof}

 Next we begin our investigation in the singularity category.
 Our result is
 \begin{theorem}
\label{mom}
(with setup as above)
Let $\Xb, \Yb, \Zb \in \K^{-,b}(\proj A)$. Set $\D_{sg} = \D_{sg}(A)$
\begin{enumerate}[\rm (1)]
  \item If $\Xb \cong \Yb$ in $\D_{sg}$ then
  \begin{enumerate}[\rm (a)]
    \item $         \alpha_{I, N}(\Xb) =   \alpha_{I, N}(\Yb)$.
    \item  $\beta_{I, V}(\Xb) =   \beta_{I, V}(\Yb)        $

  \end{enumerate}
  \item  Assume we have a triangle $s \colon \Xb \rt \Yb \rt \Zb \rt \Xb[-1]$ in $\D_{sg}$. Then
\begin{enumerate}[\rm (a)]
  \item $$\alpha_{I, N}(\Yb) \leq \max\{ \alpha_{I, N}(\Xb),  \alpha_{I, N}(\Zb)\}.$$
  \item  $$\beta_{I, V}(\Yb) \leq \max\{ \beta_{I, V}(\Xb),  \beta_{I, V}(\Zb)\}.$$
\end{enumerate}
\end{enumerate}
 \end{theorem}
 \begin{proof}
 (1) Let $\xi \colon \Xb \rt \Yb$ be an isomorphism in $\D_{sg}$.
 Note  $\xi $ can be written as a left fraction
\[
\xymatrix{
\
&\Zb
\ar@{->}[dl]_{u}
\ar@{->}[dr]^{f}
 \\
\Xb
\ar@{->}[rr]_{\xi = fu^{-1}}
&\
&\Yb
}
\]
We note that $f$ is also an isomorphism in $\D_{sg}$. We have $\cone(u), \cone(f) \in \thick(A)$.

1(a) Say $\cone(u)$ is isomorphic to a bounded  complex $\Pb$ of free $A$-modules. Say $\Pb_j = 0$ for $j \geq m$. Then for any finitely generated $A$-module we have   $H_i(\cone(u)\otimes W) = 0$ for $i \geq m$. So  for all $n \geq 0$ we have $H_i(\cone(u) \otimes N/I^{n+1}N) = 0$ for $i \geq m$.
So by \ref{triangle}(1) we get $H_i(\Zb \otimes N/I^{n+1}N) \cong H_i(\Xb \otimes N/I^{n+1}N)$ for all $ i \geq m$ and for all $n \geq 0$.
It follows that $\alpha_{I, N}(\Zb) = \alpha_{I, N}(\Xb)$. Similarly we get $\alpha_{I, N}(\Zb) = \alpha_{I, N}(\Yb)$. The result follows.

1(b) This is similar to (a).

2 We note that $s$ is isomorphic to the image in $\D_{sg}$ of a triangle $t \colon  \Xb^\prime \rt \Yb^\prime \rt \Zb^\prime \rt \Xb^\prime[-1]$ in  $\K^{-,b}(\proj A)$,
see \cite[p.\ 97]{N}.

(a) This follows from \ref{triangle}(3) and 1(a) above.

(b) This follows from \ref{triangle}(6) and 1(b) above.
 \end{proof}

\section{Proof of Theorems \ref{main-tor} and \ref{main-ext}}
Throughout set $\D_{sg} = \D_{sg}(A)$.
We first give
\begin{proof}[Proof of Theorem \ref{main-tor}]
Let $M$ be a non-free MCM $A$-module. Let $\Pb$ be the minimal projective resolution of $M$. Then $\Pb \neq 0$ in $\D_{sg}$. Let $\Qb$ be the minimal free resolution of $k$.
As $A$ is dominant we get that $\Qb \in \thick(\Pb)$.
Set $r = \limsup_{j \rt \infty} r_{I, N}^j(M)$. We have $\alpha_{I, N}(\Pb) = r$.

We now recall the definition of thick closure of $\Pb$.
We note that $\alpha_{I, N}(\Xb[m]) = \alpha_{I, N}(\Xb) $ for all $m \in \Z$. Also $\alpha_{I, N}(\Xb) \leq \alpha_{I, N}(\Xb \oplus \Yb)$.
If $\Xb \rt \Yb \rt \Zb \rt \Xb[-1]$ is a triangle in $\D_{sg}$ with  $\alpha_{I, N}(\Xb) \leq r$ and $\alpha_{I, N}(\Zb) \leq r$ then by \ref{mom}(2)(a) we get $\alpha_{I, N}(\Yb) \leq r$.
So we get $\alpha_{I, N}(\Xb) \leq r$ for all $\Xb \in \thick(\Pb)$. As $A$ is dominant we get $\Qb \in \thick(\Pb)$. So  $\alpha_{I, N}(\Qb) \leq r$. By \ref{equi-residue}, we get $r_{I, N}^2(\Omega^d(k)) = \alpha_{I, N}(\Qb)$  But by \ref{bound} we get $r \leq r_{I, N}^2(\Omega^d(k)) $. It follows that $r =  r_{I, N}^2(\Omega^d(k))$.
\end{proof}

\s The proof of  Theorem \ref{main-ext}  is similar to the proof of Theorem \ref{main-tor} and so is omitted.

 \section*{Acknowledgements}
I thank Ryo Takahashi for some discussion on dominant rings.


\begin{thebibliography} {99}







\bibitem{BH}  W. Bruns and J. Herzog,
Cohen-Macaulay Rings, revised edition,
Cambridge Studies in Advanced Mathematics, 39.
Cambridge University Press, 1998.

\bibitem{Bu}
R.-O. Buchweitz,
\emph{Maximal {C}ohen-{M}acaulay modules and {T}ate cohomology over {G}orenstein rings},
 Mathematical Surveys and Monographs, 262. American Mathematical Society, Providence, RI, 2021.



\bibitem{KT}
D.~Katz and E.~Theodorescu,
\emph{On the degree of Hilbert polynomials associated to the torsion functor},
Proc. Amer. Math. Soc, 135, (2007),  3073--3082.

\bibitem{Ko}
V.~Kodiyalam,
\emph{Homological invariants of powers of an ideal},
Proc. Amer. Math. Soc. 118 (1993), no. 3, 757–764.

 \bibitem{N}
 A.~Neeman,
 \emph{Triangulated categories},
  Annals of Mathematics Studies, 148. Princeton University Press, Princeton, NJ, 2001.


\bibitem{P1}
T.~J.~Puthenpurakal,
 \emph{Hilbert-coefficients of a Cohen-Macaulay module},
J. Algebra 264 (2003), no. 1, 82--97.




\bibitem{P3}
\bysame,
\emph{A sub-functor for Ext and Cohen-Macaulay associated graded modules with bounded multiplicity},
Trans. Amer. Math. Soc. 373 (2020), no. 4, 2567-–2589.


 \bibitem{P4}
\bysame,
 \emph{Derived functors and Hilbert polynomials over hypersurface rings},
Proc. Edinb. Math. Soc. (2) 68 (2025), no. 3, 898--908.

\bibitem{P5}
\bysame,
 \emph{Ext functors, support varieties and Hilbert polynomials over complete intersection rings},
 Preprint,  arXiv:2502.16494 
 
\bibitem{T}
R.~Takahashi,
\emph{Classifying thick subcategories of the stable category of Cohen-Macaulay modules},
Adv. Math. 225 (2010), no. 4, 2076--2116.

 \bibitem{T2}
\bysame,
\emph{Dominant local rings and subcategory classification},
Int. Math. Res. Not. IMRN 2023, no. 9, 7259--7318.

\bibitem{Theo}
E.~Theodorescu,
\emph{Derived functors and Hilbert polynomials},
Math. Proc. Camb. Phil. Soc. 132 (2002), 75--88.






\end{thebibliography}
\end{document}